\documentclass{article}

\usepackage{amssymb,amsmath,amsthm,latexsym, color}

\usepackage{graphics}

\newcommand{\T}{x_{3/2}}
\newcommand{\Tt}{y_{3/2}}

\advance \topmargin by -\headheight
\advance \topmargin by -\headsep
\evensidemargin \oddsidemargin
\theoremstyle{plain}
\newtheorem{theorem}{Theorem}

\newtheorem{proposition}[theorem]{Proposition}
\theoremstyle{definition}

\newtheorem{conjecture}[theorem]{Conjecture}

\begin{document}


\begin{center}
\vskip 1cm{\LARGE\bf The Thue-Morse sequence in base 3/2}

%

\medskip

\vskip 1cm
\large
F. M. Dekking \\
CWI and Delft University of Technology \\
Faculty EEMCS, P.O.~Box 5031\\
2600 GA Delft, The Netherlands\\
F.M.Dekking@math.tudelft.nl 
\end{center}

\vskip .2 in

\begin{abstract}
We discuss the base 3/2 representation of the natural numbers. We prove that the sum of digits function of the representation is a fixed point of a 2-block substitution on an infinite alphabet, and that this implies that sum of digits function modulo 2 of the representation is a fixed point $\T$ of a 2-block substitution on $\{0,1\}$. We prove that $\T$ is mirror invariant, and present a list of conjectured properties of $\T$, which we think will be hard to prove. Finally, we make a comparison with a variant of the base 3/2 representation, and give a general result on $p$-$q$-block substitutions.
\end{abstract}

\section{Introduction}

 A natural number $N$ is written in base 3/2 if $N$ has the form
 \begin{equation}\label{eq:exp}
    N= \sum_{i\ge 0} d_i\Big( \frac32\Big)^i,
 \end{equation}  \vspace*{-.0cm}
  with digits $d_i=0,1$ or 2.

  Base 3/2 representations are also known  as sesquinary representations of the natural numbers (see, e.g., \cite{Propp}).
  We  write these expansions as
  $${\rm SQ}(N) = d_R(N)\dots d_1(N)d_0(N)= d_R\dots d_1d_0.$$

  \noindent We have, for example, ${\rm SQ}(7)=211$, since $2\cdot(9/4)+(3/2)+1=7$.

  \medskip

  See A024629 in \cite{oeis} for the continuation of the following table.

  \medskip

 \begin{center}
 \begin{tabular}{|c|c|c|c|c|c|c|c|c|c|c|c|}
 \hline \noalign{\smallskip}
 $N$ & 0 & 1 & 2 & 3 & 4 & 5 & 6 & 7 & 8 & 9 & 10\\[.04cm]\noalign{\smallskip}
 \hline \noalign{\smallskip}
 ${\rm SQ}(N)$ & 0 & 1 & 2 & 20 & 21 & 22 & 210 & 211 & 212 & 2100 & 2101\\[.04cm]\noalign{\smallskip}
 \hline
 \end{tabular}
 \end{center}

 \medskip

\noindent  Ignoring leading 0's, the base 3/2 representation of a number $N$ is unique (see Section \ref{sec:32}).

Let for $N\ge 0$
  $$s_{3/2}(N):=\sum_{i=0}^{i=R} d_i(N)$$
be the sum of digits function of the base 3/2 expansions. We have (see A244040 in \cite{oeis})
  $$s_{3/2} = 0, 1, 2, 2, 3, 4, 3, 4, 5, 3, 4, 5, 5, 6, 7, 4, 5, 6, 5, 6, 7, 7, 8, 9, 5, 6, 7, 5, 6, 7, 7, 8, 9, 8, 9, 10,\dots$$


  \medskip

  \noindent In this note we  study the base 3/2 analogue of the Thue-Morse sequence  (where the base equals 2), i.e., the sequence (see A357448  in \cite{oeis})
  $$(\T(N)):= (s_{3/2}(N) \!\!\! \mod 2) \,= \, 0, 1, 0, 0, 1, 0, 1, 0, 1, 1, 0, 1, 1, 0, 1, 0, 1, 0, 1, 0, 1, 1, 0, 1, 1,\dots $$

  \medskip

The Thue Morse sequence is the fixed point starting with 0 of the substitution $0\rightarrow 01,\;1\rightarrow 10$. This might be called a 1-2-block substitution.
A {\it 2-3-block substitution} $\kappa$ on an alphabet $A$ replaces  blocks $ab$ of length 2 by words $\kappa(ab)$ of length 3.
Its action extends to infinite sequences $x$ by defining $\kappa:x \mapsto y$ by $y_{3k}\dots y_{3(k+1)-1}=\kappa(x_{2k}x_{2k+1})$, for $k=0,1,\dots$.

\medskip

\begin{theorem} \label{th:main} The sequence $\T$ is a fixed point of the $2$-$3$-block substitution
$$\hspace{6cm} \kappa: \; \left\{\begin{aligned}
00 & \rightarrow 010\\[-.1cm]
01 & \rightarrow 010  \\[-.1cm]
10 & \rightarrow 101  \\[-.1cm]
11 & \rightarrow 101 \end{aligned}\hspace*{16cm}\right\}$$
\end{theorem}

Theorem \ref{th:main} will be proved in Section \ref{subs:TM}.

\section{Sum of digits function and Thue-Morse in base 3/2}\label{sec:AA}

\subsection{Sum of digits function in base 3/2}\label{subs:SD}

Let  $s_{3/2}=(0, 1, 2, 2, 3, 4, 3, 4, 5, 3, 4, 5, 5, 6, 7, 4, 5,\dots)$   be the sum of digits function of the base 3/2 expansions. To describe this sequence we extend the notion of a $p$-$q$-block substitution to alphabets of infinite cardinality.

\medskip

\begin{theorem}\label{th:SD} The sequence  $s_{3/2}$ is the fixed point starting with $0$ of the $2$-$3$-block substitution given by
$$a,b \mapsto a,a+1,a+2 \quad {\rm  for\;} a = 0,1,2,... {\;\rm and\;} b = 0,1,2,....$$
\end{theorem}
	
\medskip

\begin{proof} We have $d(0)=0,d(1)=1$ and from the uniqueness of the base 3/2 expansions it follows immediately that $d(3N+r) = d(2N)+r$ for $N\ge 0$ and $r = 0, 1, 2. $

Thus $s_{3/2}(3N)=s_{3/2}(2N), s_{3/2}(3N+1)=s_{3/2}(2N)+1$,  and $s_{3/2}(3N+2)=s_{3/2}(2N)+2$. This gives the result. \end{proof}

\subsection{Thue-Morse in base 3/2}\label{subs:TM}

\noindent {\it Proof of Theorem 1.} This follows directly from Theorem \ref{th:SD} by taking $a$ and $b$ modulo 2.  $\square$

\bigskip

Although iterates of $\kappa: 00\rightarrow010, 01\rightarrow010, 10\rightarrow101, 11\rightarrow101$ are undefined, we can generate the fixed point $\T$ by iteration of a map $\kappa'$ defined by
$\kappa'(w) = \kappa(w)$ if $w$ has even length, and $\kappa'(v) = \kappa(w)$ if $v = w0$ or $v = w1$ has odd length.

The fact that the iterates of $\kappa$ are undefined causes difficulty in establishing properties of $\T$.
This is similar to the  lack of progress in the last 25 years to prove the conjectures on the Kolaskoski sequence, which is also fixed point of a 2-block substitution (cf.~\cite{Dekking-Kol}, \cite{Two-block}).
Here is a property that is open for the Kolakoski sequence, but can be proved for $\T$.

\begin{proposition} If a word $w$ occurs in $\T$, then its binary complement  $w^*$ defined  by $0^*=1, 1^*=0$, also occurs in $\T$.
\end{proposition}

\begin{proof}  First one checks this for all 16 words of length 6 that occur in $\T$. Note that then also $w^*$ occurs for all $w$ with $|w|\le 6$. Here $|w|$ denotes the length of $w$. Let $u$ be a word of length $m\ge 7$. By adding at most 3 letters at the beginning and/or end of $u$ one can obtain a word $v$ with $|v|=3n$ that occurs in $\T$ at a position 0 modulo 3. But then Theorem \ref{th:main} gives that $v=\kappa(w)$ for at least one word $w$ occurring in $\T$. The length of $w$ is $|w|=2n$. Since $\kappa(w^*) = (\kappa(w))^*$ the  result follows by induction on $m=|u|$. For example, for $|u|=m=7$, one has $|v|= 9$, and so $|w|= 6$.   \end{proof}

\medskip

Here are some conjectured properties of $\T$.

\medskip

\begin{conjecture}$\T$ is reversal invariant, i.e., if the word $w=w_1\dots w_m$ occurs in $\T$ then $\overleftarrow{w}=w_m\dots w_1$ occurs in $\T$.\end{conjecture}

\begin{conjecture} $\T$ is uniformly recurrent, i.e., each word that occurs in $\T$ occurs infinitely often, with bounded gaps.\end{conjecture}

\begin{conjecture} The frequencies $\mu[w]$ of the words $w$ occurring in $\T$ exist. Two conjectured values: $\mu[00] = 1/10,\; \mu[01] = 4/10$. \end{conjecture}

\begin{conjecture}$\mu$ is mirror invariant, i.e., $\mu[w]=\mu[w^*]$ for all words $w$.\end{conjecture}

\begin{conjecture}$\mu$ is reversal invariant, i.e., $\mu[w]=\mu[\overleftarrow{w}]$ for all words $w$.\end{conjecture}

\begin{conjecture} {\bf (J. Shallit)} The critical exponent (=largest number of repeated blocks) of $\T$ is 5. \end{conjecture}

\section{Base 3/2 and base 1/2$\cdot$3/2}\label{sec:32}

Many authors refer to the paper \cite{Akiyama-Frougny} from  Akiyama,  Frougny, and  Sakarovitch for the properties of base 3/2 expansions (see, e.g., \cite{Propp}, \cite{Rigo-Stip}).
However, the $p/q$ expansions studied in  paper  \cite{Akiyama-Frougny} are different from the
3/2 expansions that are usually considered as in Equation (\ref{eq:exp}). In paper \cite{Akiyama-Frougny}:
\begin{equation}\label{eq:frou}
    N= \sum_{i\ge 0} d_i\,\frac1q\Big( \frac{p}q\Big)^i,
 \end{equation}  \vspace*{-.0cm}
  with digits $d_i=0,1$ or 2. We write  ${\rm AFS}(N)$ for the expansion of $N$.

    \medskip

Here is the table given in  \cite{Akiyama-Frougny} for the case $p=3, q=2$.

 \begin{center}
 \begin{tabular}{|c|c|c|c|c|c|c|c|c|c|c|c|}
 \hline \noalign{\smallskip}
 $N$ & 0 & 1 & 2 & 3 & 4 & 5 & 6 & 7 & 8 & 9 & 10\\[.04cm]\noalign{\smallskip}
 \hline \noalign{\smallskip}
 ${\rm AFS}(N)$ & $\varepsilon$ & 2 & 21 & 210 & 212 & 2101 & 2120 & 2122 & 21011 & 21200 & 21202\\[.04cm]\noalign{\smallskip}
 \hline
 \end{tabular}
 \end{center}

 \medskip

These expansions will not even be found in the OEIS (at the moment).

\medskip

The situation is clarified in the paper \cite{Frougny-Klouda} by Frougny and Klouda. Here both representations are considered and called respectively base $p/q$ and base ${\small 1}/q\!\cdot\!p/q$ representations.

\medskip

A combination of the results in  \cite{Akiyama-Frougny} and \cite{Frougny-Klouda} yields a proof of the uniqueness of the base 3/2 expansions $({\rm QS}(N))$.
There is also a direct proof of uniqueness in \cite{Edgar-}, Theorem 1.1.

Note that  ${\rm AFS}(N)={\rm QS}(2N)$ for $N>0$. So uniqueness of the base 3/2 representation implies immediately uniqueness of the $1/2\!\cdot\!3/2$ representation ${\rm AFS}(N)$. This observation obviously extends to  base $p/q$.

\medskip

Next we consider the question whether also the sequence  $\Tt$, the sum of digits function modulo 2 of the base  $1/2\!\cdot\!3/2$ representation, is fixed point of a 2-block substitution.
This is indeed the case, and this 2-block substitution is given by Rigo and Stipulanti in \cite{Rigo-Stip}.

\begin{theorem}{\rm(\cite{Rigo-Stip})} \label{th:mainy} $\Tt$ is the  fixed point with prefix $00$ of the $2$-$3$-block substitution
$$\hspace{6cm} \kappa': \; \left\{\begin{aligned}
00 & \rightarrow 001\\[-.1cm]
01 & \rightarrow 000  \\[-.1cm]
10 & \rightarrow 111  \\[-.1cm]
11 & \rightarrow 110 \end{aligned}\hspace*{16cm}\right\}$$
\end{theorem}

In \cite{Rigo-Stip} the proof of Theorem \ref{th:mainy} is based on a  generalization of Cobham's theorem to what are called $\mathcal S$-automatic sequences built on tree languages with a periodic labeled signature.
Here we consider a more direct route, based on a simple closure property of $p$-$q$-block-substitutions. Recall that a coding is a letter to letter map  from one alphabet to another.

\begin{theorem} \label{th:subseq} Let $x=(x(N))$ be a  fixed point of a  $p$-$q$-block substitution. Let $r$ be a positive integer. Then $y=(x(rN))$ is a coding of a $p$-$q$-block substitution.
\end{theorem}

\begin{proof} If $x$ is a fixed point of a  $p$-$q$-block substitution, then $x$ is also fixed point of a  $pr$-$qr$-block substitution. As new alphabet, take the words of length $r$ occurring in $x$. On this alphabet, the $pr$-$qr$-block substitution induces a $p$-$q$-block substitution in an obvious way. Mapping each word of length $r$ to its first letter is a coding that gives the result. \end{proof}

\noindent {\it Alternative proof of Theorem \ref{th:mainy}}.\, We apply Theorem \ref{th:subseq} with $r=2$. The $4$-$6$-block-substitution is given by
\begin{align*}
& 0010\mapsto 010101,\;  0100\mapsto 010010,\;  0101\mapsto 010010,\; 0110\mapsto 010101,   \\
& 1001\mapsto 101010,\;  1010\mapsto 101101,\;  1011\mapsto 101101,\; 1101\mapsto 101010.
\end{align*}     
Coding $00\mapsto a,\, 01\mapsto b,\, 10\mapsto c,\, 11\mapsto d$, this induces the 2-3-block substitution
$$ ac\mapsto bbb, ba\mapsto bac, bb \mapsto bac, bc \mapsto bbb, cb \mapsto ccc, cc \mapsto cdb, cd \mapsto cdb, db \mapsto ccc.$$
If we code further $a,b\mapsto 0$, and $c,d\mapsto 1$, then we obtain $\kappa'$ from Theorem \ref{th:mainy}.

\section*{Acknowledgement} I am grateful to Jean-Paul Allouche for several useful comments.

\bibliographystyle{jis}

\bigskip

\noindent 2010 {\it Mathematics Subject Classification}: Primary 11B85, Secondary 68R15

\medskip

\noindent  \emph{Keywords:} Base 3/2, Thue-Morse sequence, sum of digits, two-block substitution

\end{document}